\documentclass[a4paper,10pt]{article}
\usepackage[utf8]{inputenc}
\usepackage[english]{babel}
\usepackage{comment,xcomment,amssymb,amsmath,color}
\usepackage{rotating,xypic,array}
\usepackage{latexsym,mathtools,amsthm}
\usepackage{booktabs}  
\usepackage{url} 
\usepackage[section]{placeins}
\usepackage{multirow,rotating,graphicx}
\usepackage[a4paper]{geometry}
\usepackage[colorlinks, citecolor=blue]{hyperref}
\usepackage{enumitem}
\newlist{compactenum}{enumerate}{4}
\setlist[compactenum,1]{nolistsep}
\usepackage{xspace}
\usepackage[dvipsnames]{xcolor}
\title{Ad-invariant metrics on nonnice nilpotent Lie algebras}
\author{Diego Conti, Viviana del Barco and Federico A. Rossi}

\makeatletter
\def\namedlabel#1#2{\begingroup
   \def\@currentlabel{#2}%
   \label{#1}\endgroup
}
\makeatother

\newtheorem{theorem}{Theorem}[section]
\newtheorem{lemma}[theorem]{Lemma}

\newtheorem{proposition}[theorem]{Proposition}
\theoremstyle{definition}

\newtheorem{example}[theorem]{Example}

\theoremstyle{remark}
\newtheorem{remark}[theorem]{Remark}

\newcommand{\R}{\mathbb{R}}
\newcommand{\im}{\mathrm{Im}\,}         
\newcommand{\lie}[1]{\mathfrak{#1}}     
\newcommand{\g}{\lie{g}}

\newcommand{\N}{\mathbb{N}}

\newcommand{\hook}{\lrcorner\,}

\newcommand{\id}{\mathrm{Id}}   

\newcommand{\Span}[1]{\operatorname{Span}\left\{#1\right\}}

\DeclareMathOperator{\diag}{diag}

\DeclareMathOperator{\ad}{ad}

\DeclareMathOperator{\Tr}{tr}

\newcolumntype{C}{>{$}c<{$}}
\newcolumntype{L}{>{$}l<{$}}
\newcolumntype{R}{>{$}r<{$}}

\newcommand{\st}{\;\mid\;}          

\newcommand{\mg}{\mathfrak{g}}

\newcommand{\mn}{\mathfrak{n}}
\newcommand{\mz}{\mathfrak{z}}
\newcommand{\mk}{\mathfrak{k}}

\newcommand{\mh}{\mathfrak{h}}

\newcommand{\lela}{\left\langle}
\newcommand{\rira}{\right\rangle}
\newcommand{\bil}{\lela \,,\,\rira}

\newcounter{rowno}
\begin{document}
\maketitle

\begin{abstract}
We proved in previous work that all real nilpotent Lie algebras of dimension up to $10$ carrying an ad-invariant metric are nice, i.e. they admit a nice basis in the sense of Lauret, Nikolayevsky and Will. In this paper we show by constructing explicit examples that nonnice irreducible nilpotent Lie algebras admitting an ad-invariant metric exist for every dimension greater than $10$ and every nilpotency step greater than $2$. In the way of doing so, we introduce a method to construct Lie algebras with ad-invariant metrics called the single extension, as a parallel to the well-known double extension procedure.
\end{abstract}

\renewcommand{\thefootnote}{\fnsymbol{footnote}}
\footnotetext{\emph{MSC class 2020}: \emph{Primary} 17B30; \emph{Secondary} 53C30, 53C50, 17B01}
\renewcommand{\thefootnote}{\arabic{footnote}}

\section*{Introduction}
A nondegenerate scalar product of any signature on a Lie algebra is called an ad-invariant metric if inner derivations are skew-symmetric. From the geometrical point of view this is a natural condition to consider. Indeed, it implies that the induced metric on the associated simply connected Lie group is bi-invariant and thus its Levi-Civita connection (as its curvature) is entirely determined by the Lie bracket~\cite{oneill}. In particular, ad-invariant metrics on nilpotent Lie algebras induce Ricci-flat Lie groups.

Lie algebras admitting a Riemannian ad-invariant metric are products of a compact Lie algebra and an abelian factor.
For general signature, the classification of Lie algebras admitting ad-invariant metrics is an open problem, which has been solved when restricted to particular subfamilies  \cite{BaumKath,dBO12,dBOV,FavreSantharoubane:adInvariant,Kath_Olbrich_2006,
Kath:NilpotentSmallDim,DuongPinczonUshirobira:ANewInvariant}. In fact,
some of these classifications also list all possible ad-invariant metrics, up to isometry; it turns out that many solvable Lie algebras only admit one ad-invariant metric (see \cite{BaBe07,ContiDelBarcoRossi:Uniqueness,MeRe93} for some results on the uniqueness of ad-invariant metrics).

In particular, real nilpotent Lie algebras with an ad-invariant metric are classified up to dimension $10$ in~\cite{Kath:NilpotentSmallDim}. In~\cite{ContiDelBarcoRossi:Uniqueness}, we showed that the irreducible Lie algebras appearing in this classification admit a unique ad-invariant metric (up to sign); as a step in the proof, we proved that all of them admit a \emph{nice basis}. Nice bases are mostly studied in the context of Einstein Riemannian metrics on solvmanifolds. They were first introduced in \cite{LauretWill:EinsteinSolvmanifolds}, and received their (now widely accepted) name in \cite{Nikolayevsky}. Lie algebras with a nice basis, called nice Lie algebras in this paper, possess a strong algebraic structure, which can be encoded in a special kind of directed graph. Nice nilpotent Lie algebras have been classified up to dimension $6$ in \cite{FernandezCulma:classification2,LauretWill:diagonalization} and up to dimension $9$ in~\cite{ContiRossi:Construction} (see also~\cite{KadiogluPayne:Computational} for the particular case where the Nikolayevsky derivation is simple and the root matrix is surjective); these classifications show that in small dimension most, but not all, nilpotent Lie algebras are nice.

A natural question arises: does every nilpotent Lie algebra carrying ad-invariant metrics fulfill the nice condition?

In this paper we answer this question in the negative. More precisely, we consider the problem of determining which steps of nilpotency and dimensions can be attained by a nonnice nilpotent Lie algebra admitting an ad-invariant metric. Since the direct sum of two Lie algebras admitting an ad-invariant metric also admits an ad-invariant metric, we restrict our attention to irreducible Lie algebras, i.e. those which cannot be decomposed as a direct sum.

We show that for every $d>10$ there exists an irreducible nonnice nilpotent Lie algebra of dimension $d$ admitting an ad-invariant metric. In addition, we prove that for every $s>2$ there exists an irreducible nonnice nilpotent Lie algebra of step $s$ admitting an ad-invariant metric.

We obtain these Lie algebras by three different constructions. Two are classical methods to construct Lie algebras with ad-invariant metrics, namely the double extension procedure \cite{FavreSantharoubane:adInvariant,MedinaRevoy} and considering the ad-invariant metric given by the natural pairing on any cotangent Lie algebra. The third one is new; we dub it \emph{single extension} due to its close resemblance to double extensions.

The main difficulty that we faced in the process of producing the examples is to actually prove that a given Lie algebra is not nice. Recall that a Lie algebra is nice if it admits a basis, called a nice basis, such that the structure constants satisfy certain vanishing conditions (see Section~\ref{sec:first} for the precise definition). One of the methods we use to show nonexistence of such a basis takes into account a canonical derivation that exists on any Lie algebra, known as the Nikolayevsky, or pre-Einstein derivation (see~\cite{Nikolayevsky}). We observe that, if a nice basis exists, one can assume that the eigenspaces of the Nikolayevsky derivation are generated by nice basis elements. This fact combined with  a counting argument allows us to prove that a nice basis does not exist for free nilpotent Lie algebras of step greater than two on more than two generators (Theorem~\ref{thm:rude}). Since these free nilpotent Lie algebras do not admit ad-invariant metrics~\cite{dBO12}, we consider their cotangents.

It is easy to see that a nice basis $\{e_i\}$ of a Lie algebra $\g$ determines a nice basis of the cotangent $T^*\g$, obtained by adding to $\{e_i\}$ the elements of its dual basis. Thus, the cotangent of any nice Lie algebra is nice. However the converse is not clear at all, since a cotangent Lie algebra $T^*\g$ may admit a nice basis which does not contain a basis of $\g$.

We prove that this converse statement does hold in the special case of free nilpotent Lie algebras. Thus, the cotangent of a nonnice free nilpotent Lie algebra is again nonnice; this gives rise to nonnice nilpotent Lie algebras of any step $>2$ with an ad-invariant metric.

In order to obtain nonnice nilpotent Lie algebras with an ad-invariant metric in all dimensions greater than $10$, we construct two explicit examples in dimensions $11$ and $12$, and then build on the $12$-dimensional example using single and double extensions. In the $11$-dimensional case, nonexistence of a nice basis is proved by studying the eigenspaces of the Nikolayevsky derivation and a direct inspection of the Lie algebra structure; for the others, we use a different argument using quotients and the low dimensional classification of nice Lie algebras given in~\cite{ContiRossi:Construction}.

In both cases, some extra work is needed in order to prove that the constructed Lie algebras are in fact irreducible. To this end, we prove two technical results; one concerns a special class of positively-graded Lie algebras, the other a particular type of metric Lie algebra with a distinguished nondegenerate subspace which encodes part of the Lie algebra structure in a suitable sense.

Even though all results in this paper are stated over $\R$, many of the constructions work over more general fields.

We have not been able to produce nonnice $2$-step nilpotent Lie algebras carrying ad-invariant metrics with the techniques of this paper. Large classes of nonnice nilpotent Lie algebras of step two are proved to exist by a counting argument in \cite[Example~3]{Nikolayevsky}; as the construction is not explicit, it is not clear if any of the resulting Lie algebras admit an ad-invariant metric.
Whether such examples exist is left as an open problem.

\section{Nonnice Lie algebras with an ad-invariant metric in low dimensions}
\label{sec:first}

In this section we present an example of an $11$-dimensional nilpotent Lie algebra which admits ad-invariant metrics and is not nice. Furthermore, we introduce a $12$-dimensional example with similar characteristics, together with some technical results which will be of use in the next section to construct higher dimensional examples.

We start by recalling the definition of nice Lie algebras. A Lie algebra is called \emph{nice} if it admits a basis $\{e_1,\dotsc, e_n\}$ such that every Lie bracket of the form $[e_i,e_j]$ is a multiple of an element of the basis and, if we denote by $\{e^1,\dotsc, e^n\}$ the dual basis and by $d$ the Chevalley-Eilenberg differential, every $e_i\hook de^j$  has to be a multiple of some element of the dual basis (see \cite{LauretWill:EinsteinSolvmanifolds,Nikolayevsky}). In this case, we say that $\{e_1,\dotsc, e_n\}$ is a \emph{nice basis}. A  Lie algebra is \emph{nonnice} if it does not admit any nice basis.

Recall from~\cite{Nikolayevsky} that every Lie algebra $\mg$ admits a semisimple derivation $N$ such that for every derivation $D$ of $\mg$, one has
\begin{equation}\label{eq:nikdef}
\Tr(ND)=\Tr(D).
\end{equation}
This derivation $N$ is unique up to automorphisms of $\mg$, and (any) such $N$ is called the Nikolayevsky (or pre-Einstein) derivation of $\mg$. Computing the Nikolayevsky derivation is a straightforward but tedious linear computation; for the examples appearing in this paper, we used the program~\cite{Conti:Gleipnir}.

When the eigenvalues of the Nikolayevsky derivation have multiplicity one, eigenvectors form a nice basis (see~\cite{LauretWill:diagonalization}). By way of converse,
when the Lie algebra $\mg$ is nice, then the Nikolayevsky derivation can be assumed to be diagonal relative to the nice basis (see~\cite[Proof of Theorem~3]{Nikolayevsky} and \cite[Theorem~3.1]{Payne:Applications}); in other words, $\mg$ admits a nice basis constituted by eigenvectors of the Nikolayevski derivation. These two facts indicate that nonnice Lie algebras should be sought among Lie algebras whose Nikolayevsky derivation has large eigenspaces, and that nonexistence of a nice basis can be proved by only considering bases that reflect the eigenspace decomposition.
\begin{proposition}
\label{prop:nonnice11}
The $11$-dimensional nilpotent Lie algebra
\begin{equation}\label{eq:not}
\g=(0,0,- e^{12},0,\frac{1}{2}  e^{24},e^{13}-\frac{1}{2} e^{24},e^{14}-e^{23},-2 e^{16}+e^{27},e^{26}-e^{17}+e^{25},e^{18}-e^{46}-e^{37}-e^{45},e^{28}-2 e^{36}-e^{47})
\end{equation}
has an ad-invariant metric
\[g=-e^1\odot e^{11}+ e^3\odot e^8+e^4\odot e^9+e^2\odot e^{10}-2e^5\otimes e^5+2e^6\otimes e^6 +e^7\otimes e^7,\]
and it does not admit any nice basis.
\end{proposition}

The notation in which the structure of $\mg$ is given above will be used throughout the paper. It means that $\mg^*$ has a basis $\{e^1, \ldots, e^{11}\}$ with the Chevalley-Eilenberg differential $d$ of each element determined by~\eqref{eq:not}, i.e. $de^1=de^2=0$, $de^3=-e^1\wedge e^2$ and so on. Notice that in \eqref{eq:not}, $e^{12}$ represents the two-form $e^1\wedge e^2$; when indices take more than one digit, they will be separated by a comma, as in $e^{1,10}$, standing for $e^1\wedge e^{10}$. The symbol $\odot$ represents symmetric product, i.e. $e^1\odot e^{11}=e^1\otimes e^{11}+e^{11}\otimes e^1$.
\begin{proof}
It is easy to check that $g$ is ad-invariant. Assume for a contradiction that a nice basis exists. We say that a subspace of $\g$ is nice if it is spanned by nice basis elements. We say that an element of $\g$ is nice if it is a multiple of an element of the basis. It is clear that for any nice element $v\in\g$, $\ker \ad v$ and $\im \ad v$ are nice. In addition, we can assume that the eigenspaces of the Nikolayevsky derivation are nice. Canonical computations show that a Nikolayevsky derivation of $\mg$ in~\eqref{eq:not}, relative to the basis $\{e_i\}$, is given by
\[\frac{33}{119}\diag(1,1,2,2,3,3,3,4,4,5,5);
\]
the eigenspaces are given by
\[\Span{e_1,e_2},\ \Span{e_3,e_4},\ \Span{e_5,e_6,e_7},\ \Span{e_8,e_9},\ \Span{e_{10},e_{11}}.\]
The nice basis of $\g$ contains two elements of $\Span{e_1,e_2}$; their Lie bracket is a nonzero multiple of $[e_1,e_2]$. Then
$e_3=[e_1,e_2]$ is nice, $[e_3,\Span{e_1,e_2}]=\Span{e_6,e_7}$ is nice and $e_5$ is nice, as it spans the intersection of $\Span{e_5,e_6,e_7}$ with $\ker \ad e_3$.
From $\im \ad e_5$ we get that $e_9$ and $e_{10}$ are nice, and from $\ker \ad e_5$ that so is $e_1$. Then $e_6=[e_1,e_3]$ is nice, hence $e_{11}=\frac12[e_3,e_6]$ is nice, and so is $e_4$, generating the space \[\{v\in\Span{e_3,e_4}\mid [v,e_6]\subset\Span{e_{10}}\}.\]
It follows that each of $e_4,e_5,e_6$ belongs to the nice basis up to a multiple; however, $[e_4,e_5]=e_{11}=[e_4,e_6]$, which is absurd.
\end{proof}

It was observed in~\cite{LauretWill:diagonalization} that any nice basis of a Lie algebra $\mg$ is adapted to the lower central series (LCS): if $\g^k$ denotes the $k$-th element of the LCS, every $\g^k$ is spanned by elements of the nice basis, and therefore every $\g/\g^k$ is nice. Notice that we will also denote the elements  $\g^1$ and $\g^2$ of the LCS by $\g'$ and $\g''$, respectively.
A similar result holds for the upper central series (UCS):
\begin{proposition}
\label{lemma:UCS}
Let $\{e_1,\dotsc, e_n\}$ be a nice basis of a Lie algebra $\g$. Then every ideal $C_k(\g)$ of the UCS is spanned by elements of the nice basis; in particular, each quotient $\g/C_k(\g)$ is nice.
\end{proposition}
\begin{proof}
The proof is by induction on $k$.

For $k=1$, $C_k(\g)$ is the center $\mz(\mg)$ of $\mg$. Since $\ker \ad e_i$ is spanned by elements of the nice basis for all $i$, the same holds for the intersection $\mz(\g)=\bigcap_{i=1}^n \ker \ad e_i$. Thus,  $\mz(\g)$ is spanned by elements of the nice basis.

For the inductive step, assume that $C_k(\mg)$ has a nice basis $\{e_1,\dotsc, e_r\}$ and let $\pi:\mg\to \g/C_k(\mg)$ be the canonical projection. Then the images $\pi(e_{r+1}),\dotsc, \pi(e_n)$ form a nice basis of  $\g/C_k(\mg)$. In particular, the center of $\g/C_k(\mg)$ is spanned by elements of the nice basis, call them $\pi(e_{r+1}),\dotsc, \pi(e_m)$. Then $C_{k+1}(\mg)$ is spanned by $e_1,\dotsc, e_m$, proving the inductive step.
\end{proof}

The fact that the 11-dimensional Lie algebra in Proposition \ref{prop:nonnice11} is nonnice and the fact that all the examples listed in~\cite{Kath:NilpotentSmallDim} are nice Lie algebras allow us to state the following:
\begin{proposition}
\label{prop:existreducible}
All the nilpotent Lie algebras of dimension $n\leq 10$ admitting ad-invariant metrics are nice. For any $n \geq 11$ there exists a nonnice Lie algebra of dimension $n$ admitting ad-invariant metrics.
\end{proposition}
\begin{proof}
Nilpotent Lie algebras of dimension $\leq 10$ admitting ad-invariant metrics were classified by Kath in~\cite{Kath:NilpotentSmallDim} and they were shown to be nice in~\cite{ContiDelBarcoRossi:Uniqueness} (see Table~$1$ therein).

For any dimension $n=11+k\geq 11$ take $\mg\oplus \R^k$, where $\mg$ is the $11$-dimensional Lie algebra in Proposition~\ref{prop:nonnice11}.

Observe that $\mg\oplus\R^k$ is nonnice as well: call $\pi$ the projection $\mg\oplus\R^k\to \mg$. Assume for a contradiction that $\mg\oplus\R^k$ has a nice basis $\{v_1,\dotsc, v_{11+k}\}$;
then it is adapted to both the UCS and the LCS, so we can assume up to reordering that the center is generated by $v_{9},\dotsc, v_{11+k}$, and $v_{9},v_{10},v_{11}$ lie in $\mg'$. Then $\pi(v_1),\dotsc, \pi(v_{11})$ are linearly independent, and they form a nice basis of $\mg$, which is absurd.
\end{proof}
The construction of Proposition~\ref{prop:existreducible} suggests a natural question, namely whether \emph{irreducible} nonnice nilpotent Lie algebras with an ad-invariant metric exists in any dimension $n\geq 11$. For $n=11$, we can answer in the affirmative by observing that if the Lie algebra $\mg$ of Proposition~\ref{prop:nonnice11} were reducible, say $\mh=\mh_1\oplus\mh_2$, then both $\mh_1$ and $\mh_2$ would be nice by the first part of Proposition~\ref{prop:existreducible}, which is absurd. For $n>11$, we will answer this question in the affirmative in the rest of this  section and the next. The proof will make use of the Lie algebra constructed in \cite[Example~2.11]{ContiDelBarcoRossi:Uniqueness}, which we now recall.
\begin{example}\label{ex:NilpWeakSolDim12}
Consider the $12$-dimensional nilpotent irreducible Lie algebra $\mh$ given by:
\begin{multline*}
(0,0,e^{14},e^{12},e^{47}+e^{28}+e^{39}+e^{10,11},e^{18}+e^{4,10}-e^{3,11}-e^{9,11},\\
- e^{19}-e^{2,11}+e^{4,11},-e^{17}+e^{2,10}-e^{3,11},e^{13}-e^{2,11},e^{24}+e^{1,11},0,-e^{1,10}-e^{23}-e^{29}-e^{34}).
\end{multline*}
It admits the ad-invariant metric
\[h=-e^1\odot e^5+e^2\odot e^6-e^3\odot e^7-e^4\odot e^8-e^9\otimes e^{9}+e^{10}\otimes e^{10}+e^{11}\odot e^{12}\]
and one can easily check that
\begin{equation}\label{eq:scdh}
\mh'=\Span{e_4}\oplus \mh'', \;\mh''=\Span{e_3,e_5,e_6,e_7,e_8,e_9,e_{10},e_{12}} \mbox{ and } \mz(\mh)=\Span{e_5,e_6,e_{12}}.
\end{equation}

This Lie algebra appeared in~\cite{ContiDelBarcoRossi:Uniqueness} as an example of a nilpotent Lie algebra admitting more than one ad-invariant metric up to automorphisms, but only one up to automorphisms and rescaling. This is linked to the fact that the Nikolayevsky derivation of $\mh$ is zero, unlike the $11$-dimensional example of Proposition~\ref{prop:nonnice11} (see \cite[Proposition~2.21]{ContiDelBarcoRossi:Uniqueness}). We will denote by $\lie{n}=\mh/\mz(\mh)$ the quotient of $\mh$ by its center $\mz(\mh)$, whose structure equations are
\[(0,0,e^{14},e^{12},-e^{17}-e^{29}+e^{49},
-e^{15}+e^{28}-e^{39},e^{13}-e^{29},e^{24}+e^{19},0).\]
The nilpotent Lie algebra $\mn$ will be shown below to be nonnice, and this is what makes $\mh$ an effective base for constructing irreducible nonnice Lie algebras with an ad-invariant metric in higher dimensions.
\end{example}

In order to show that the $12$-dimensional Lie algebra in Example~\ref{ex:NilpWeakSolDim12} is nonnice, we will need the following:
\begin{lemma}\label{lem:Center12DimNNnice}
Both the $9$-dimensional Lie algebra $\lie{n}$ from Example~\ref{ex:NilpWeakSolDim12} with structure equations
\[(0,0,e^{14},e^{12},-e^{17}-e^{29}+e^{49},
-e^{15}+e^{28}-e^{39},e^{13}-e^{29},e^{24}+e^{19},0)\]
and the $10$-dimensional Lie algebra $\tilde{\lie{n}}$ with structure equations
\[
(0,0,e^{14},e^{12},- e^{17}-e^{29}+e^{49},-e^{15}+e^{28}-e^{39},e^{13}-e^{29},e^{24}+e^{19},0,
e^{19})
\]
are not nice.
\end{lemma}
\begin{proof}
The Lie algebra $\mn$ is $6$-step nilpotent. The dimensions of the ideals in the UCS of $\lie{n}$ are $(\dim C_i(\lie{n}))_{i=0}^5\allowbreak=\allowbreak(1,3,4,6,7,9)$, and the dimensions of the ones in the LCS are $(\dim\lie{n}^i)_{i=0}^5\allowbreak=\allowbreak(9,6,5,3,2,1)$. Looking at the list of nice Lie algebras of dimension $9$ of~\cite{ContiRossi:Construction}, we see that there are two Lie algebras that have the same dimensions as $\mn$ of the LCS and UCS, namely
\begin{equation}\label{eq:18ab}
\begin{gathered}
\texttt{965321:18a}\qquad(0,0,0,-e^{12},e^{14},e^{24}, e^{15} + e^{23},e^{17} + e^{34},e^{18} + e^{26} + e^{35})\\
\texttt{965321:18b}\qquad(0,0,0,-e^{12},e^{14},e^{24},-e^{15} + e^{23},e^{17} + e^{34},e^{18} + e^{26} + e^{35}).
\end{gathered}
\end{equation}
So if $\lie{n}$ is nice, it must be isomorphic to one of these two Lie algebras. However, the Nikolayevsky derivation of $\lie{n}$ is zero, whilst \texttt{965321:18a,b} have nonzero Nikolayevsky derivation (see~\cite{ContiRossi:NiceNilsolitons}). Indeed, one can check that their Nikolayevsky derivation, in the basis satisfying~\eqref{eq:18ab}, is given by
\[\frac6{29}\diag(1,2,3,3,4,5,5,6,7).\]
So $\mn$ cannot be isomorphic to any of those Lie algebras, and $\lie{n}$ is not nice.

For $\tilde{\lie{n}}$ we follow the same proof. First note that the Nikolayevsky derivation of $\tilde{\lie{n}}$ vanishes.

A computer assisted classification, using the algorithm of \cite{ContiRossi:Construction}, \cite{demonblast}, shows that the nice Lie algebras of dimension $10$ with UCS $(\dim \tilde{\lie{n}}^i)_{i=0}^5=(2,4,5,7,8,10)$ and LCS $(\dim C_i(\tilde{\lie{n}}))_{i=0}^5=(10,7,5,3,2,1)$ are the ones in Table~\ref{Table10Dim}. However, none of them has a zero Nikolayevsky derivation, so $\tilde{\lie{n}}$ is not nice.
{\setlength{\tabcolsep}{2pt}
\begin{table}[ht]
\centering
{\small
\caption{\label{Table10Dim} $10$-dimensional nice nilpotent Lie algebras with UCS $(2,4,5,7,8,10)$ and LCS $(10,7,5,3,2,1)$, and their Nikolayevsky derivation}
\begin{tabular}{ C C}
\toprule
 \text{Lie algebra} & \text{Nikolayevsky derivation}\\
\midrule
 0,0,0,e^{12},e^{13},e^{14},e^{35},e^{16}+e^{23},e^{18}-e^{34},e^{19}-e^{36} & \frac{2}{37}(26, 23, 21, 20, 17, 14, 12, 11, 9, 3)\\
 0,0,0,e^{12},e^{13},e^{14},e^{15},e^{23}+e^{16},-e^{34}+e^{18},-e^{36}+e^{19} & \frac{2}{137}(101, 88, 75, 65, 62, 52, 49, 39, 36, 13)\\
 0,0,0,e^{12},e^{13},e^{14},e^{24}+e^{35},e^{16}+e^{23},e^{18}-e^{34},e^{19}-e^{36} & \frac{24}{137}(8, 7, 7, 6, 5, 4, 4, 3, 3, 1)\\
 0,0,0,e^{12},e^{13},e^{14},-e^{35}+e^{24},e^{23}+e^{16},-e^{34}+e^{18},-e^{36}+e^{19} & \frac{24}{137}(8, 7, 7, 6, 5, 4, 4, 3, 3, 1)\\
 0,0,0,e^{12},e^{13},e^{14},e^{24}+e^{15},-e^{23}+e^{16},e^{34}+e^{18},e^{36}+e^{19} & \frac{4}{19}(7, 6, 5, 5, 4, 4, 3, 3, 2, 1)\\
 0,0,0,e^{12},e^{13},e^{14},e^{15}+e^{24},e^{16}+e^{23},e^{18}-e^{34},-e^{36}+e^{19} & \frac{4}{19}(7, 6, 5, 5, 4, 4, 3, 3, 2, 1)\\
 0,0,0,e^{12},e^{13},e^{14},e^{35},e^{16}+e^{23},e^{18}-e^{34},e^{19}+e^{24}-e^{36} & \frac{27}{173}(9, 8, 7, 7, 6, 5, 4, 4, 3, 1)\\
 0,0,0,e^{12},e^{13},e^{14},e^{15},e^{23}+e^{16},-e^{34}+e^{18},e^{24}-e^{36}+e^{19} & \frac{26}{161}(9, 8, 7, 6, 5, 5, 4, 4, 3, 1)\\
 0,0,0,e^{12},e^{13},e^{14},e^{24},e^{23}+e^{16},-e^{34}+e^{18},e^{27}-e^{36}+e^{19} & \frac{4}{19}(7, 6, 5, 5, 4, 4, 3, 3, 2, 1)\\
 0,0,0,e^{12},e^{13},e^{14},e^{24},e^{16}+e^{23},e^{18}-e^{34},e^{19}-e^{27}-e^{36} & \frac{4}{19}(7, 6, 5, 5, 4, 4, 3, 3, 2, 1)\\
\end{tabular}
}
\end{table}
}
\end{proof}

\begin{proposition}\label{pr:WeakSolDim12NotNice}
There exists an irreducible $12$-dimensional Lie algebra with an ad-invariant metric that is not nice.
\end{proposition}
\begin{proof}
We show that the $12$-dimensional Lie algebra $\mh$ in Example~\ref{ex:NilpWeakSolDim12} is neither nice nor reducible.

Assume for a contradiction that it is nice;
by Proposition~\ref{lemma:UCS}, the quotient $\mh/\mz(\mh)$ is nice. This quotient is isomorphic to the Lie algebra $\mn$ shown to be nonnice in Lemma~\ref{lem:Center12DimNNnice}, giving an absurd.

Now suppose for a contradiction that $\mh$ is reducible, say $\mh=\mh_1\oplus\mh_2$. Then at least one of $\mh_1$, $\mh_2$ is not nice, we can assume $\mh_1$. Then $\mh_1$ also admits an ad-invariant metric (see \cite[Lemma~3.2]{ZhZh01}), so by Proposition~\ref{prop:existreducible} $\dim\mh_1\geq 11$. Since $\mz(\mh)\subset \mh'$, we have that $\R$ is not a direct factor in $\mh$, and therefore $\mh_1=\mh$ and $\mh_2=0$; thus, $\mh$ is irreducible.
\end{proof}

\section{Constructions in higher dimensions}
In this section we construct nonnice irreducible Lie algebras with an ad-invariant metric in any dimension $\geq13$. These examples will be obtained from those in Section~\ref{sec:first} using two constructions.

The first construction is the double extension, introduced in~\cite{MedinaRevoy}. Given a Lie algebra $\mh$ with an ad-invariant metric $h$ and a skew-symmetric derivation $D$, the double extension of $(\mh,h)$ by $D$ is a Lie algebra $\mg$ endowed with an ad-invariant metric $g$. As a vector space, $\mg=\mh\oplus\Span{e,z}$, with the Lie bracket given by
\begin{gather*}
[X,Y]=[X,Y]_\mh + h(DX,Y)z, \quad X,Y\in\mh;\\
[e,X]=DX,\quad  X\in\mh;\\
\ad z=0;
\end{gather*}
the metric is given by
\begin{gather*}
g(X,Y)=h(X,Y), \quad g(X,z)=0, \quad g(X,e)=0, \qquad X,Y\in\mh;\\
g(z,z)=0=g(e,e);\\
g(z,e)=1.
\end{gather*}
The second construction is the following:
\begin{proposition}
\label{prop:single}
Let $(\mh,[\,,\,]_\mh)$ be a Lie algebra with an ad-invariant metric $h$ and let $D\colon\mh\to\mh$ be a skew-symmetric derivation of rank two. Then the vector space $\g=\mh\oplus\Span{U}$ has a Lie algebra structure, defined by a Lie bracket $[\,,\,]$, and an ad-invariant metric $g$ given by
\begin{gather}\label{eq:strSE}
[X,Y]=[X,Y]_{\mh}+h(DX,Y)U, \quad [U,X]=DX,\\
g(X,Y) = h(X,Y), \quad g( X,U)=0, \quad g(U,U)=1,\label{eq:metSE}
\end{gather}
where $X,Y\in\mh$.

Moreover, if $D$ is noninner and satisfies $ \mz(\mh)\subset\ker D$, then
\begin{equation}
\label{eqn:singleexteq2}
\mg'=\mh'\oplus \Span{U}, \quad \mz(\mg)=\mz(\mh).
\end{equation}
\end{proposition}

\begin{proof}
Let $p\colon\mh\to\mh/\ker D$ be the projection. Since $D$ is skew-symmetric, we can define an alternating map
\[\gamma\colon \Lambda^3\mh\to\mh/\ker D, \quad \gamma(X,Y,Z)=h(DX,Y)p(Z)+h(DY,Z)p(X)+h(DZ,X)p(Y),\]
and it is easy to see that $\gamma(X,Y,Z)=0$ if $X\in\ker D$. Therefore, we can write
\[\gamma=p^*\tilde\gamma, \quad \tilde\gamma\colon\Lambda^3(\mh/\ker D)\to \mh/\ker D.\]
Since $D$ has rank two, $\tilde\gamma$ is zero. This implies
\begin{equation}
\label{eqn:singleext}
h(DX,Y)Z+h(DY,Z)X+h(DZ,X)Y\in \ker D.
\end{equation}
In order to prove that $[\,,\,]$ defines a Lie algebra structure on $\mg$, we need to verify the Jacobi identity on vectors $X,Y,Z$ and $X,Y,U$, where $X,Y,Z\in\mg$. We compute
\[[[X,Y],Z]=[[X,Y]_{\mh},Z]+h(DX,Y)[U,Z]=[[X,Y]_\mh,Z]_\mh +h(D[X,Y]_\mh,Z)U
+h(DX,Y)DZ.
\]
So
\begin{multline*}
[[X,Y],Z]+[[Y,Z],X]+[[Z,X],Y]\\
=[[X,Y]_\mh,Z]_\mh+[[Y,Z]_\mh,X]_\mh+[[Z,X]_\mh,Y]_\mh
+\bigl(h(D[X,Y]_\mh,Z)+h(D[Y,Z]_\mh,X)+h(D[Z,X]_\mh,Y)\bigr)U\\
\shoveright{+h(DX,Y)DZ+h(DY,Z)DX+h(DZ,X)DY}\\
=\bigl(h(-[X,Y]_\mh,DZ)+h(-[Y,Z]_\mh,DX)+h(D[Z,X]_\mh,Y)\bigr)U
+D\bigl(h(DX,Y)Z+h(DY,Z)X+h(DZ,X)Y\bigr)\\
=\bigl(h(Y,[X,DZ]_\mh)+h(-Y,[Z,DX]_\mh)+h(D[Z,X]_\mh,Y)\bigr)U=0,
\end{multline*}
where we have used that $\lie h$ is a Lie algebra, $D$ is a skew-symmetric derivation, $h$ is ad-invariant and~\eqref{eqn:singleext} holds.

Similarly,
\begin{multline*}
 [[X,Y],U]+[[Y,U],X]+[[U,X],Y]
 =-D([X,Y]_\mh)-[DY,X]+ [DX,Y]\\
 =-D([X,Y]_\mh)+[X,DY]- [Y,DX]
 =h(DX,DY)U-h(DY,DX)U=0.
\end{multline*}

Finally, ad-invariance of the metric follows from
\begin{align*}
g( [X,Y],Z) &= g( [X,Y]_{\mh},Z) + g( h(DX,Y)U,Z) = -g( Y,[X,Z]_{\mh}) = -g( Y,[X,Z]);\\
g( [X,Y],U) &= g( h(DX,Y)U,U) =h(DX,Y)=-h(X,DY)=-g( X,DY)\\
&=-g( X, [U,Y]) =g( X,[Y,U]);\\
g( [X,U],U) &=g( -DX,U) = 0 = -g( U,[X,U]);\\
g( [U,X],Y) &=g( DX,Y) = h(DX,Y)=-h(X,DY)=-g( X,[U,Y]).
\end{align*}

For the last part of the statement, assume that $D$ is noninner and $\mz(\mh)\subset\ker D$.
The inclusion $\mz(\mh)\subset\mz(\mg)$ follows from $D|_{\mz(\mh)}=0$ and~\eqref{eq:strSE}. To show the other inclusion, let $\pi\colon\mg\to\mh$ denote the projection. If $X+\lambda U$ is in $\mz(\mg)$, then  $\pi\circ(\ad X+\lambda D)=0$; since $D$ is noninner, this implies $\lambda=0$, $X\in\mz(\mh)$. Finally, taking the orthogonal complement of $\mz(\mh)=\mz(\mg)$ in $(\mg, g)$ and using~\eqref{eq:metSE} we get
\[
\mg'=\mz(\mg)^{\perp_g}=\mz(\mh)^{\perp_g}=\mh'\oplus\Span{U},\]
proving~\eqref{eqn:singleexteq2}.
\end{proof}
Given a Lie algebra $\mh$ endowed with an ad-invariant metric $h$ and a skew-symmetric derivation of rank two $D$, we will refer to the Lie algebra constructed in Proposition~\ref{prop:single} as the \emph{single extension} of $(\mh,h)$ by the derivation $D$. Notice that a single extension can be obtained from a double extension $\mh\oplus\Span{e,z}$ defined by the same derivation $D$ by formally identifying
$e$ with $z$. This identification is not a quotient by an ideal.

\begin{example}
Take $\mh=\R^4$, with the ad-invariant metric $e^1\odot e^3+e^2\odot e^4$,  and
\[D=e^1\otimes e_4-e^2\otimes e_3.\]
Then the single extension is given by
\[[e_1,e_2]=U,\quad [U,e_1]=e_4,\quad [U,e_2]=-e_3,\]
resulting in the unique irreducible $5$-dimensional nilpotent Lie algebra carrying an ad-invariant metric. The double extension of $\mh$ by $D$ is the unique irreducible $6$-dimensional nilpotent Lie algebra carrying an ad-invariant metric.
\end{example}

In order to prove that the extensions we construct are irreducible, it will be convenient to introduce another definition.

Let $(\mh,h)$, $(\mg,g)$ be Lie algebras with an ad-invariant metric. We say that an injective linear map $\iota\colon\mh\to\mg$ is a \emph{mirage} of $(\mh,h)$ in $(\mg,g)$ if the next conditions are satisfied:
\begin{enumerate}[label=(M\arabic*)]
\item \label{item:D1} $g(\iota(v),\iota(w))=h(v,w)$ for all $v,w\in\mh$;
\item \label{item:D2} for any $X\in\mh'$ and $Y\in\mh$;
$[\iota(X),\iota(Y)]=\iota([X,Y])$;
\item \label{item:D3} for any $X\in\mh'$ and $Y\in\iota(\mh)^\perp$, $[\iota(X),Y]=0$;
\item \label{item:D4} $\mg''=\iota(\mh'')$;
\item \label{item:D5} for any nondegenerate ideal $I\supset\iota(\mh)$ of $\mg$, $I^\perp$ is central.
\end{enumerate}
It is straightforward to check that for any Lie algebra with ad-invariant metric $(\mh,h)$, the identity map on $\mh$ is a mirage of $(\mh,h)$ in itself.

The following technical result identifies a class of single extensions that behave well relative to mirages:
\begin{lemma}
\label{lemma:singleextgpp}
Let $(\mh,h)$, $(\mg,g)$ be Lie algebras with an ad-invariant metric, let $\iota\colon\mh\to\mg$ be a mirage, and assume $\mz(\mh)\subset\mh'$. Let $D\colon\mg\to\mg$ be a rank-two skew-symmetric map which is not an inner derivation, satisfying $D|_{\iota(\mh')}=0$, $\im D\subset\mg''$, and assume that there is an $X\in\iota(\mh)$ such that $[X,\mg]\subset\iota(\mh)$ and $DX\neq0$. Then, $D$ is a derivation of $\mg$ and the single extension $(\mk,k)$ of $\mg$ by $D$ verifies that $\iota\colon\mh\to\mk$ is a mirage of $(\mh,h)$ in $(\mk,k)$. In addition,
\begin{equation}
\label{eqn:singleexteq}
\mk'=\mg'\oplus\Span{U}, \quad \mz(\mk)=\mz(\mg).
\end{equation}
\end{lemma}
\begin{proof}
We first prove that $D$ is a derivation of $\g$. Observe that $\im D$ is orthogonal to $\iota(\mh')\subset\ker D$ inside $\mg$. Because of~\ref{item:D1}, this implies $\im D\subset \iota(\mz(\mh))\oplus \iota(\mh)^\perp$; however, by hypothesis and~\ref{item:D4}, we have $\im D\subset
\g''=\iota(\mh'')\subset\iota(\mh)$, which gives
$\im D\subset \iota(\mz(\mh))$.
By assumption $\mz(\mh)\subset\mh'$, then~\ref{item:D2} and~\ref{item:D3} imply that any element of $\iota(\mz(\mh))$ commutes with $\mg$, so
\[\im D\subset \mz(\mg).\]
Since $D$ is skew-symmetric, it vanishes on $\mg'$, which is the orthogonal $\mz(\mg)$; finally note that any linear map $D\colon\mg\to\mg$ such that $D|_{\mg'}=0$ and $\im D\subset \mz(\mg)$ is a derivation of $(\mg,g)$. In particular, the single extension $(\mk,k)$ of $\mg$ by $D$ is well defined, with $\mk=\mg\oplus\Span{U}$.

Moreover, $\im D$ is contained in $\mg''$, hence in $\g'$; since $D$ is skew-symmetric, it vanishes on $\mz(\mg)$, the orthogonal of $\mg'$. Therefore, $D$ satisfies the conditions of the last part of Proposition~\ref{prop:single}, so we get~\eqref{eqn:singleexteq} as a consequence of~\eqref{eqn:singleexteq2}.

To show that $\iota\colon \mh\to\mk$ is a mirage of $(\mh,h)$ in $(\mk,k)$, note that~\ref{item:D1} holds by construction of the single extension and the fact that $\iota\colon \mh\to\mg$ is a mirage. Moreover, \ref{item:D2} and~\ref{item:D3} follow from the fact that $D$ is zero on $\iota(\mh')$, and by~\eqref{eqn:singleexteq} we have
\[\mk''=[\mg\oplus\Span{ U},\mg'\oplus\Span{ U}]=\mg''+\im D=\mg'',\]
giving~\ref{item:D4}.

Finally, suppose that $I$ is a nondegenerate ideal in $\mk$ containing $\iota(\mh)$, and choose $X$ in $\iota(\mh)\setminus \ker D$ such that $[X,\mg]\subset\iota(\mh)$. Then for any $Y\in\mg$,
\[[X,Y]_\mk=[X,Y]+g(DX,Y)U=g(DX,Y)U \mod \iota(\mh);\] choosing $Y$ in $\mg$ such that $g(DX,Y)\neq0$, we see that $I$ must contain $U$. Now, $I\cap \mg$ is a nondegenerate ideal of $\mg$. Since $\iota\colon\mh\to\mg$ is a mirage, the orthogonal complement of $I\cap \mg$ in $\mg$ is central in $\mg$. As $\mk$ and $\mg$ have the same center, the orthogonal complement of $I$ in $\mk$ is also central, and \ref{item:D5}~holds as well, so the result follows.
\end{proof}
Given two Lie algebras with an ad-invariant metric $(\mh,h)$ and $(\mg,g)$, we will say that $(\mh,h)$ is a \emph{mirage} in $(\mg,g)$ if $\mh$ is a vector subspace of $\mg$ and the inclusion map $\iota$ is a mirage; then the metric on $\mh$ is the restriction of the metric on $\mg$, but the restriction of the Lie bracket $[\,,\,]_\g$ to $\mh$ will not generally coincide with the Lie bracket $[\,,\,]_\mh$ of $\mh$.

Given a mirage $(\mh,h)$ in $(\mg,g)$,
condition~\ref{item:D1} implies that the restriction of $g$ to $\mh$ is nondeg\-enerate, so we have a vector space splitting
\[\mg=\mh\oplus\mh^\perp.\]
If we  decompose a vector $Y\in\mg$
as $Y=Y_{\mh}+Y_{\mh^\perp}$, conditions~\ref{item:D2} and~\ref{item:D3} can be restated as
\begin{equation}
\label{eq:xyh}
[X,Y]_\mg=[X,Y_{\mh}]_\mh, \quad X\in\mh', Y\in\mg.
\end{equation}

\begin{proposition}
\label{prop:diegoirreducible}
Let $(\mh,h)$ be the $12$-dimensional Lie algebra with the ad-invariant metric of Example~\ref{ex:NilpWeakSolDim12}. Let $\mg$ be a Lie algebra satisfying $\mz(\g)\subset\g'$ with an ad-invariant metric $g$, and suppose that there is a mirage of $(\mh,h)$ in $(\mg,g)$. Then $\mg$ is irreducible.
\end{proposition}
\begin{proof}
Recall that $\mh$ has structure constants
\begin{multline*}
(0,0,e^{14},e^{12},e^{47}+e^{28}+e^{39}+e^{10,11},e^{18}+e^{4,10}-e^{3,11}-e^{9,11},\\
- e^{19}-e^{2,11}+e^{4,11},-e^{17}+e^{2,10}-e^{3,11},e^{13}-e^{2,11},e^{24}+e^{1,11},0,-e^{1,10}-e^{23}-e^{29}-e^{34}),
\end{multline*}
and the metric is
\[h=-e^1\odot e^5+e^2\odot e^6-e^3\odot e^7-e^4\odot e^8-e^9\otimes e^{9}+e^{10}\otimes e^{10}+e^{11}\odot e^{12}.\]
We shall denote $\{e_1,\ldots, e_{12}\}$ the basis of $\mh$ dual to $\{e^i,\ldots, e^{12}\}$, the one above giving the structure constants.

Suppose that $\mg$ decomposes as the sum of two ideals, $\mg=\mk_1\oplus \mk_2$. We can assume that $\mk_1$ and $\mk_2$ are orthogonal relative to the ad-invariant metric $g$ (see \cite[Lemma~3.2]{ZhZh01}).

By~\ref{item:D4}, we have $\mh''=\mg''=\mk_1''\oplus\mk_2''$. Since $(\mh'')'=\Span{e_5}$ is one-dimensional, we can assume that $\mk_1''$ contains $e_5$ and $\mk_2''=0$.

Using~\ref{item:D2} we compute:
\begin{align*}
[e_3,\mg]&=\Span{e_5,e_6+e_8,e_9,e_{12}}, &
  [e_4,\mg]&=\Span{e_3,e_5,e_6,e_7,e_{10},e_{12}},\\
  [e_7,\mg]&=\Span{e_5,e_8}, &
  [e_8,\mg]&=\Span{e_5,e_6}, \\
  [e_9,\mg]&=\Span{e_5,e_6,e_7,e_{12}}, &
  [e_{10},\mg]&=\Span{e_5,e_6,e_8,e_{12}}.
\end{align*}
Since $\mk_1$ contains $e_5$, it follows that $\mk_2$ is orthogonal to $e_5$. Thus, writing the generic element of $\mk_2$ as $a_1e_1+\dots + a_{12}e_{12}+v$ with $a_i\in\R$, $v\in\mh^\perp$, the component $a_1$ is zero. By~\ref{item:D2} and~\ref{item:D3}, it follows that
\[
  [e_7,\mk_2]\subset \Span{e_5}\cap \mk_2=0,\quad
  [e_8,\mk_2]\subset \Span{e_5}\cap \mk_2=0, \\
\]
Since $\mg=\mk_1\oplus\mk_2$, this implies
\[[e_7,\mk_1]=[e_7,\mg], \quad [e_8,\mk_1]=[e_8,\mg],\]
i.e. $\mk_1$ contains $e_5,e_6$ and $e_8$.
Since $\mk_2$ is orthogonal to $\mk_1$, its elements have no component along $e_1,e_2$ and $e_4$, so
\begin{gather*}
[e_3,\mk_2]\subset\Span{e_5,e_6+e_8},\quad [e_9,\mk_2]\subset\Span{e_5,e_6}, \quad
  [e_{10},\mk_2]\subset \Span{e_5}.
\end{gather*}
This implies that $\mk_2$ commutes with $e_3$, $e_9$ and $e_{10}$, so its elements have no component along $e_3, e_9, e_{11}$, i.e. $\mk_2$ is orthogonal to $e_7, e_9,e_{12}$, which must belong to $\mk_1$.

Then $[e_4,\mk_2]\subset\Span{e_5,e_6}$,
 i.e. $\mk_2$ also commutes with $e_4$, showing that its elements have no component along $e_7,e_{10}$; consequently, $\mk_1$ contains $e_3$ and $e_{10}$. Summing up,
\[\mk_1\supset \Span{e_3,e_5,e_6,e_7,e_8,e_9,e_{10},e_{12}}=\mh''.\]
Therefore $\mk_2=\mk_1^\perp$ is contained in $\Span{e_5,e_6,e_8,e_{12}}\oplus\mh^\perp$. It follows that the projection $\pi^\perp\colon\g\to\mh^\perp$ is injective on $\mk_2$, and maps $\mk_2\subset\mh'\oplus\mh^\perp$ isomorphically on a nondegenerate ideal of $\mg$. By condition~\ref{item:D5}, $\pi^\perp(\mk_2)$ is central. Since the center of $\mg$ is contained in $\mg'$, the metric on $\pi^\perp(\mk_2)$ is zero, so $\mk_2$ is trivial.
\end{proof}

Let $\{v_1,\dotsc, v_n, w_1,\dotsc, w_n\}$ be a basis of $\R^{2n}$, with dual basis $\{ v^1,\dotsc,  v^n, w^1,\dotsc, w^n\}$, and consider the neutral metric
\begin{equation}
\label{eq:neutral}
\bil =  v^1\odot  w^1+\dots +  v^n\odot  w^n.
\end{equation}
For $n=1$, we will write $\{\hat v,\hat w\}$ instead of $\{v_1,w_1\}$.

Fixing $(\mh,h)$ as in Proposition~\ref{prop:diegoirreducible} (and using the notation therein), for any $n\geq 0$ we consider the Lie algebra $\mh\oplus\R^{4n}$, together with the ad-invariant metric defined by $h$ and $\bil$ as in~\eqref{eq:neutral}, being the factors orthogonal. The following linear map defines a skew-symmetric derivation
\[D_{4n}\colon \mh\oplus \R^{4n}\to \mh\oplus \R^{4n},\quad
 e_1\mapsto e_6,\quad e_2\mapsto e_5,\quad v_{2i}\mapsto w_{2i-1},\quad  v_{2i-1}\mapsto-w_{2i}.
\]
Similarly, on $\mh\oplus\R^2\oplus\R^{4n}$ endowed with the product metric we have a skew-symmetric derivation
\begin{gather*}
D_{4n+2}\colon \mh\oplus \R^2\oplus \R^{4n}\to \mh\oplus \R^2\oplus \R^{4n},\\
e_1\mapsto e_6+\hat w,\quad e_2\mapsto e_5,\quad v_{2i} \mapsto w_{2i-1},\quad  v_{2i-1}\mapsto -w_{2i},\quad \hat{v}\mapsto e_5.
\end{gather*}
Notice that for any $n\geq 0$ and $l=0,2$,  $D_{4n+l}$ vanishes on $\mh'$, $D_{4n+l}^2=0$ and $\im D_{4n+l}\subset \mz(\mh)\oplus\mh^\bot$ (see~\eqref{eq:scdh}).
\begin{theorem}
\label{thm:12andhigher}
Let $(\mh,h)$ be as in Proposition~\ref{prop:diegoirreducible}, and let $f\colon \mh\to\mh$ be a linear map satisfying
\[f(e_1)=e_{12},\ f(e_{11})=e_5,\ f(e_i)=0,\quad i\neq 1,11.\]
Let $\mg_{13}$ be the single extension of $\mh$ by $f$ and for $n\geq0$:
\begin{itemize}
\item let $\mg_{14+4n}$ be the double extensions of $\mh\oplus\R^{4n}$ by $D_{4n}$;
\item let $\mg_{16+4n}$ be the double extensions of $\mh\oplus\R^2\oplus\R^{4n}$ by $D_{4n+2}$;
\item let $\mg_{15+4n}$ be the single extension of $\mg_{14+4n}$ by $f$ (extended to zero on $\mh^\perp$);
\item let $\mg_{17+4n}$ be the single extension of $\mg_{16+4n}$ by $f$ (extended to zero on $\mh^\perp$).
\end{itemize}
Then, for each $k\geq 13$, $\mg_k$ is a nonnice irreducible Lie algebra of dimension $k$ with an ad-invariant metric.
\end{theorem}
\begin{proof}
We first assume that $k$ is even, and show that $\mh$ is a mirage in $\mg_{k}$. Along the proof we will denote by $D$ the derivation $D_{k-14}$.

The double extension of $\mh\oplus \R^{4n}$ and $\mh\oplus \R^2\oplus \R^{4n}$ by $D$ defines a metric and a Lie algebra structure on
the vector space $\mg_k=\mh\oplus \mh^\perp$, where the metric restricted to $\mh$ coincides with $h$ and with $\mh^\perp$ equal to either $\R^{4n}\oplus \Span{e,z}$ or $(\R^2\oplus \R^{4n})\oplus\Span{e,z}$. Thus,~\ref{item:D1} is trivially satisfied.

The nonzero Lie brackets in $\mg_k$ are as follows:
\begin{align}
[X,Y]&=[X,Y]_\mh + h(DX,Y)z, && X,Y\in\mh;\label{b1}\\
[X,v]&=h(DX,v)z, && X\in\mh, v\in \R^{4n}, v\in\R^2\oplus \R^{4n};\label{b2}\\
[v,w]&=h(Dv,w)z,&& v,w\in \R^{4n}, v,w\in\R^2\oplus \R^{4n};\label{b3}\\
[e,X]&=DX,&&  X\in\mh;\label{b4}\\
[e,v]&=Dv, && v\in \R^{4n}, v\in\R^2\oplus \R^{4n}.\label{b5}
\end{align}
Conditions~\ref{item:D2} and~\ref{item:D3} follow from \eqref{b1}, \eqref{b2}, \eqref{b4}, and the fact that $D|_{\mh'}=0$.

By \eqref{b1}--\eqref{b5}, it is clear that
\[\mg_k'+\Span{z} = \mh' + \im D+\Span{z}.\]
This implies that $\mg_k''=[\mg_k,\mh']+[\mg_k,\im D]$. By~\ref{item:D2} and~\ref{item:D3}, $[\mg_k,\mh']=[\mh,\mh']=\mh''$. Moreover, $[\mg_k,\im D]=0$ since $\im D\subset \mz(\mh)\oplus \mh^\bot$ and $D^2=0$.
Therefore $\mg_k''=\mh''$ and thus~\ref{item:D4} is satisfied.

Now let $I$ be a nondegenerate ideal of $\mg_k$ containing $\mh$. Then $I$ contains $z$ by \eqref{b2}, so $I^\bot$ is an ideal contained in $\R^{4n}\oplus\Span{z}$, or $\R^2\oplus\R^{4n}\oplus\Span{z}$. In addition  $I^\perp$ does not contain $z$ because $I\cap I^\bot=0$ by nondegeneracy. Hence, given any $v+az\in I^\perp$ with $v\in \R^2\oplus\R^{4n}$ or $\R^{4n}$, by \eqref{b2}, \eqref{b3}, we get $[X,v]=0=[v,w]$ and thus $g(Dv,X)=g(Dv,w)=0$, for all $X\in\mh$, $w\in \R^{4n}$ or $w\in\R^2\oplus\R^{4n}$. Therefore, $I^\perp \subset\ker D+\Span{z}$. Thus, $I^\perp$ is abelian, giving~\ref{item:D5}.

This shows that the inclusion of $\mh$ in $\mg_k$ is a mirage, which by Proposition~\ref{prop:diegoirreducible} implies that $\mg_{k}$ is irreducible.
To show that $\mg_k$ is nonnice, first notice that
\[
\mz(\mg_k)=(\mz(\mh)+\im D)\oplus \Span{z}.
\]This implies that $\mg_k/\mz(\mg_k)$ is isomorphic to a direct sum of ideals, one of which is an abelian factor and the other one is isomorphic to the nonnice 9-dimensional Lie algebra $\mh/\mz(\mh)=\lie{n}$ in Lemma~\ref{lem:Center12DimNNnice}. By Proposition~\ref{lemma:UCS} we get that $\mg_k$ is also nonnice. This proves the part of the statement for $k$ even.

For $k$ odd, let $\mg_k$ be the single extension of $\mg_{k-1}$ by $f$, where $\mg_{12}=\mh$. We observe that $f$ is a noninner derivation of $\mg_{k-1}$, skew-symmetric, zero on $\mh'\supset \mz(\mh)$, and its image is contained in $\mh''=\mg_{k-1}''$. Moreover, for $X=e_{11}$ we have
\[[X,\mg_{k-1}]\subset\iota(\mh), \quad fX\neq0.\]
Thus, Lemma~\ref{lemma:singleextgpp} applies and $\mh$ is a mirage in $\mg_k$, and thus $\mg_k$ is irreducible. In addition, $\mg_k/\mz(\mg_k)$ is isomorphic to an abelian factor plus the  central extension of $\mh/\mz(\mh)$
by the two-cocycle $e^1\wedge e^{11}$; the latter is isomorphic to the $10$-dimensional Lie algebra $\tilde{\lie{n}}$ which was shown to be nonnice in Lemma~\ref{lem:Center12DimNNnice}. Again, $\mg_k$ is nonnice by Proposition~\ref{lemma:UCS}.
\end{proof}

\section{Cotangents of free nilpotent Lie algebras}
\label{sec:free}

In this section we show that the cotangent of any free nilpotent Lie algebra of step $\geq 3$ is irreducible and not nice. This shows that for every $s\geq 3$, there exists a nonnice irreducible Lie algebra of nilpotency index $s$ carrying an ad-invariant metric.

We start by introducing general sufficient conditions for a given Lie algebra to be irreducible and not nice. These results will be later applied to free nilpotent Lie algebras and their cotangents.

Recall that when the Lie algebra $\mg$ is nice, then the Nikolayevski derivation can be assumed to be diagonal relative to the nice basis (see~\cite{Payne:Applications}); in other words, $\mg$ admits a nice basis constituted by eigenvectors of the Nikolayevski derivation. In particular, this fact allows to  provide criteria for the niceness of a given Lie algebra, as the following result shows.

\begin{lemma}
\label{lemma:freenotnice}
Let $\g$ be a Lie algebra with Nikolayevski derivation $N$, and let $W$ be an eigenspace of $N$ of dimension  $m$. If $\mg$ is nice, then
\begin{equation*}\label{eq:Wm}
\dim [W,[W,W]]\leq \dfrac{m (-4 + 3 m + m^2)}6.
\end{equation*}
\end{lemma}
\begin{proof}
We argue as in \cite[Example~4]{Nikolayevsky}. Since $N$ is the Nikolayevsky derivation, we can assume that $W$ is spanned by elements of the nice basis $\{X_1,\dotsc, X_m\}$. It follows that
\[[[W,W],W]=\Span{[[X_i,X_j],X_k]\st 1\leq i,j,k\leq m}.\]
By the Jacobi identity, if $i$, $j$ and $k$ are distinct, then the elements
\[[[X_i,X_j],X_k],\ [[X_j,X_k],X_i],\ [[X_k,X_i],X_j]\]
are linearly dependent; since each triple bracket is a multiple of an element of the nice basis, they are all three contained in a one-dimensional subspace. In addition, $i$ and $j$ can be assumed to be distinct in $ [[X_i,X_j],X_k]$.

Therefore,
\[\dim [[W,W],W]\leq m(m-1)+\binom{m}3=\frac{m (-4 + 3 m + m^2)}6,\]
as we wanted to prove.
\end{proof}

Given a  Lie algebra $\mg$, its cotangent Lie algebra $T^*\g$ is defined as a semidirect product $\g\ltimes\g^*$, where $\mg^*$ is an abelian ideal, and the action of $\mg$ on $\mg^*$ is given by
\[[X,\alpha]=\ad_X^*\alpha=-\alpha\circ\ad_X, \quad\text{ for all }X\in \mg,\; \alpha\in \mg^*.\]
Any cotangent Lie algebra has a canonical ad-invariant metric $g$ induced by the pairing of $\g$ with $\g^*$, namely
\begin{equation}\label{eqn:CanonicalCotangentMetric}
g(X+\alpha,Y+\beta)=\alpha(X)+\beta(Y),\quad\text{ for all }X,Y\in \mg,\; \alpha,\beta\in \mg^*.
\end{equation}

For any endomorphism $S\colon T^*\mg\to T^*\mg$, we denote by $S^*$ its metric adjoint with respect to $g$. In particular, if $P\colon T^*\mg\to T^*\mg$ denotes the projection onto $\mg^*$, it is easy to show that $P+P^*=\id$.

Finally, notice that every linear map $T\colon\mg\to \mg$ defines an endomorphism of $T^*\mg$, which we also denote by $T$,  by extending it by zero on $\mg^*$.

\begin{lemma}
\label{lemma:nonnice}
Let $\g$ be a nonnice nilpotent Lie algebra with Nikolayevsky derivation $ N$. If, for every eigenvalue $\lambda$ of $N$, $2-\lambda$ is not an eigenvalue,  then  $T^*\g$ is not nice.
\end{lemma}
\begin{proof}
Let $N$ be the Nikolayevsky derivation of $\g$. Recall from~\cite[Proposition~5.20]{ContiDelBarcoRossi:Uniqueness} that $T^*\g$ has Nikolayevsky derivation
\begin{equation}\label{eq:tildeN}
\tilde N=a( N- N^*+2P),
\end{equation}
where $a$ is a constant, $0<a<1$.

Let $\lambda_1,\dotsc, \lambda_n$ denote the distinct eigenvalues of $\tilde N$ and let $X+\alpha$ be an eigenvector of $\tilde N$, associated to the eigenvalue $\mu$. Equation~\eqref{eq:tildeN} implies
\[
\mu X=aNX \qquad \text{ and }\qquad \mu\alpha =a(2\id-N^*)\alpha.
\]
If $X$ and $\alpha$ are both nonzero, one gets that $\mu/a=\lambda_j=(2-\lambda_k)$ is an eigenvalue of $N$, for some $j,k=1,\ldots, n$, contradicting the hypothesis. Therefore, every eigenvector belongs to either $\mg$ or $\mg^*$, and thus $\mg$ and $\mg^*$ are invariant by $\tilde N$.

Assume for a contradiction that $T^*\g$ is nice. Up to an automorphism, we may assume that $\tilde N$ is diagonal relative to the nice basis. In particular, the subalgebra $\mg$ is spanned by a subset of such nice basis, itself forming a nice basis of $\mg$, which is absurd.
\end{proof}

In order to state conditions that guarantee irreducibility of graded Lie algebras, we introduce some notation. Given a Lie algebra $\g$ and a subspace $U$, we denote by $C(U)=\{X\in\lie \g\mid (\ad X)|_U=0\}$ the centralizer of $U$ in $\g$, which is not in general an ideal.
\begin{lemma}
 \label{lemma:irreducible}
Let $\g=\bigoplus_{k=1}^r \mg_k$ be a positively graded Lie algebra such that
\begin{enumerate}
\item \label{it:vzv} the only subspaces $V\subseteq \mg_1$ such that $\mg_1\subseteq V+C(V)$ are $V=0$ and $V=\mg_1$,
\item \label{it:zv} if $\mg_k\nsubseteq \mz(\mg)$, the only subspace $V\subset \mg_k$ such that $[V,\mg_1]=0$ is $V=0$,
\item \label{it:gncomm} $\mz(\mg)\subseteq \mg'$.
\end{enumerate}
Then $\g$ is irreducible.
\end{lemma}

\begin{proof}
For each $k$, denote by $\pi_k$ the projection on $\mg_k$, and assume $\g$ decomposes as the direct sum of ideals, $\mg=\mh\oplus\mk$. We trivially have
\begin{equation}
 \label{eqn:g1sum}
\mg_1=\pi_1(\mh)+\pi_1(\mk).
\end{equation}
For any $X,Y$ in $\mg$,  we have
\begin{equation}
\label{eqn:pi2XY}
\pi_2([X,Y])=[\pi_1(X),\pi_1(Y)].
\end{equation}
Then $\pi_1(\mh)$ and $\pi_1(\mk)$ also commute; indeed, if $X\in\mh$ and $Y\in\lie k$ then
\[0=\pi_2([X,Y])=[\pi_1(X),\pi_1(Y)].\]
Thus, $\pi_1(\mk)\subset C(\pi_1(\mh))$ and~\eqref{eqn:g1sum}, together with~\ref{it:vzv} in the hypothesis, imply that either $\pi_1(\mh)=0$ or $\pi_1(\mh)=\mg_1$; the same holds for $\pi_1(\mk)$.
We cannot have $\pi_1(\mh)=\g_1=\pi_1(\mk)$
because~\ref{it:gncomm} implies that $\g_1\nsubseteq \mz(\mg)$, so the case $k=1$ of~\ref{it:zv} gives $[\g_1,\g_1]\neq0$.
Therefore, without loss of generality, we can assume that $\pi_1(\mh)$ is zero, and by~\eqref{eqn:g1sum}, $\pi_1(\mk)=\g_1$.

We will prove by induction that for every $k=1,\ldots, r$,
\begin{equation}
\label{eq:ind}
\mg_k\nsubseteq \mz(\mg)\Longrightarrow\pi_k(\mh)=0.
\end{equation}
By our assumption,~\eqref{eq:ind} holds for $k=1$ since $\pi_1(\mh)=0$.

Suppose~\eqref{eq:ind} holds for $k-1$ and let $Y\in\mh$. Then,
\[
Y=\sum_{j=1}^{r}\pi_j(Y),\quad \text{ where }\quad \sum_{j=1}^{k-1}\pi_j(Y)\in \mz(\mg).
\]
Hence, using again that $\mk$ and $\mh$ commute, we get
\[
0=\pi_{k+1}([X,Y])=\sum_{i=1}^{k+1}[\pi_i(X),\pi_{k+1-i}(Y)]=[\pi_1(X),\pi_k(Y)],\qquad \forall X\in\mk.
\]
Therefore, $\pi_k(\mh)$ commutes with $\mg_1$, which by~\ref{it:zv} implies that either $\mg_k\subseteq \mz(\mg)$ or $\pi_k(\mh)=0$. So~\eqref{eq:ind} follows for all $k=1, \ldots, r$.

As a consequence, $\mh\subseteq \mz(\mg)$ which implies both $\g'=\mh'\oplus\mk'=\mk'$, and  $\mh\subseteq\g'$ by~\ref{it:gncomm}. So $\mh=0$ and $\mg$ is irreducible.
\end{proof}

We are now ready to introduce the objects of study of the present section: free nilpotent Lie algebras. We shall briefly recall some of their main structural properties. We refer the reader to \cite{GrGr90,Se92} for further information on these Lie algebras.

Let $\mn_{m,s}$ denote the free nilpotent Lie algebra of degree $s$ over $m\geq 2$ generators. We have a grading
\begin{equation}\label{eq:gradnds}
\lie {n}_{m,s}=W_1\oplus W_2\oplus \dots \oplus  W_s
\end{equation}
where, if $\{X_1,\ldots, X_m\}$ is a generator set, $W_1=\Span{X_1,\dotsc, X_m}$.
The $k$-th term of the central descending series of $\mn_{m,s}$ is given by
\[
(\mn_{m,s})^k=W_{k+1}\oplus \dots \oplus W_{s},\qquad k=0,\ldots, s-1.\]

For each $k=1,\ldots, s$, $\dim W_{k}=d_m(k)$ satisfies
\begin{equation}
 \label{eqn:recurrenceformula}
 kd_m(k)=m^k-\sum_{\ell<k,\ \ell| k} \ell d_m(\ell).
\end{equation}
In particular, for any $m$ one has
\begin{equation}
 \label{eqn:dmsmalls}
\begin{gathered}
d_m(1)=m,\quad d_m(2)=m(m-1)/2,\quad d_m(3)=m(m^2-1)/3,\\
d_m(4)=m^2(m^2-1)/4, \quad d_m(5)=m(m^4-1)/5.
\end{gathered}
\end{equation}

An explicit basis of $\mn_{m,s}$, known as the Hall basis, can be constructed through bases of each $W_1,\ldots, W_s$ with the methods of~\cite{Ha50}.
\begin{example}
\label{example:hall}
Let $\mn_{2,5}$ be the free $5$-step nilpotent Lie algebra on the two generators $e_1,e_2$.

The Hall basis $\{e_i\}_{i=1}^{14}$ can be explicitly written as
\begin{gather*}
[e_2,e_1]=:e_3\\
[[e_2,e_1],e_1]=[e_3,e_1]=:e_4, \quad[[e_2,e_1],e_2]=[e_3,e_2]=:e_5\\
[[[e_2,e_1],e_1],e_1]=[e_4,e_1]=:e_6, \quad [[[e_2,e_1],e_1],e_2]=[e_4,e_2]=:e_7,\quad
[[[e_2,e_1],e_2],e_2]=[e_5,e_2]=:e_8,\\
[[[[e_2,e_1],e_1],e_1],e_1]=[e_6,e_1]=:e_{9},\quad
[[[[e_2,e_1],e_1],e_1],e_2]=[e_6,e_2]=:e_{10},\\
[[[[e_2,e_1],e_1],e_2],e_2]=[e_7,e_2]=:e_{11},\quad
[[[[e_2,e_1],e_2],e_2],e_2]=[e_8,e_2]=:e_{12},\\
[e_4,e_3]=:e_{13},\quad  [e_5,e_3]=:e_{14}.
\end{gather*}
The basis is adapted to the grading~\eqref{eq:gradnds}, namely
\[\mn_{2,5}=W_1\oplus \dots \oplus W_5=\Span{e_1,e_2}\oplus \Span{e_3}\oplus \Span{e_4,e_5}\oplus\Span{e_6,e_7,e_8}\oplus\Span{e_9,\dots, e_{14}}.\]
\end{example}

\begin{lemma}
\label{lemma:cotangentirreducible}
If $s\geq 2$, then $T^*\mn_{m,s}$  is irreducible.
\end{lemma}

\begin{proof}
Consider the positive grading of $T^*\mn_{m,s}$ obtained by
extending the natural grading~\eqref{eq:gradnds} of $\mn_{m,s}$, and assigning degree $2s+1-i$ to $W_i^*$. We will show that $T^*\mn_{m,s}$, with this grading, satisfies the hypotheses of Lemma~\ref{lemma:irreducible}.

To see that~\ref{it:vzv} holds, pick a nonzero proper subspace $V$ of $W_1$. If we let $I=\bigoplus_{i=3}^{2s} W_i$, the quotient $T^*\mn_{m,s}/I$ is isomorphic to $\lie{n}_{m,2}$. For any $0\neq X\in V$, the centralizer of  $X+I$ in $T^*\mn_{m,s}/I\cong \lie{n}_{m,2}$ intersects $W_1+I$ in $\Span{X+I}$. Therefore, $C(\Span{X})\cap W_1\subset\Span{X}$. This shows that $V+C(V)$ does not contain $W_1$.

Notice that for every $i=1,\ldots, s-1$, $W_i$ is not contained in the center and $[X,W_1]=0$ for $X\in W_i$, implies $X=0$. For the dual spaces in the grading, one can easily check that $W_1^*$ is contained in the center; moreover,
for every $i=2,\ldots, s$, if $\alpha\in W_i^*$ satisfies $[\alpha,W_1]=0$, then
\[
\alpha([X,Y])=0, \qquad \text{ for all }X\in W_1,\;Y\in W_{i-1}  .
\]
Since $W_i=[W_1,W_{i-1}]$, this implies that $\alpha=0$. Therefore, $W_i^*\nsubseteq \mz(T^*\mn_{m,s})$ and the only subspace of $W_i^*$ commuting with $W_1$ is the trivial space.

From the above, we get that~\ref{it:zv} in Lemma~\ref{lemma:irreducible} holds and also that
$\mz(T^*\mn_{m,s})= W_s\oplus W_1^*$. In addition, one can easily check that
\begin{equation}\label{eq:derivcot}
(T^*\mn_{m,s})'\subseteq W_2\oplus \cdots \oplus W_s\oplus W_{s-1}^*\oplus \cdots \oplus W_1^*
\end{equation}
and, since $\dim (T^*\mn_{m,s})'=\dim T^*\mn_{m,s}-\dim \mz(T^*\mn_{m,s})$, we get that the sets in~\eqref{eq:derivcot} are actually equal. Hence,~\ref{it:gncomm} in Lemma~\ref{lemma:irreducible} is satisfied and the Lie algebra $T^*\mn_{m,s}$ is irreducible.
\end{proof}

On the graded Lie algebra $\mn_{m,s}$, let $D=\sum_{k=1}^s k\pi_k$, where $\pi_k$ is the projection on the space of degree $k$.
\begin{lemma}
\label{lemma:nikfree}
The Nikolayevski derivation of $\mn_{m,s}$ is $N=\lambda D$, where $\lambda$ satisfies
\begin{equation}
 \label{eqn:horrible}
 \sum_{k=1}^{s} kd_m(k)(k\lambda-1)=0.
\end{equation}
In particular, $N$ has positive eigenvalues.
\end{lemma}
\begin{proof}
Any linear map $W_1\to\mn_{m,s}$ extends to a unique derivation of $\mn_{m,s}$. In particular, the inclusion map extends to the derivation $D$ defined above. Moreover, every derivation of $\mn_{m,s}$ splits as a linear combination of $D$ and a traceless derivation  $f$ such that $\Tr(Df)=0$ (see~\cite[Proposition 2.4]{BdlC14}).
Therefore, if we set
\[N=\lambda D, \qquad \Tr N=\Tr N^2,\]
we have that $\Tr(Nf)=\Tr f$ for all $f$, i.e. $N$ is the Nikolayevski derivation.

Explicitly, $\lambda$ must satisfy
\[\lambda \sum_{k=1}^s kd_m(k)= \lambda^2\sum_{k=1}^s k^2d_m(k);\]
this implies that either~\eqref{eqn:horrible} holds or $\lambda=0$. However the latter cannot hold, since $\Tr (ND)=\Tr(D)\neq0$. Furthermore,~\eqref{eqn:horrible} implies $\lambda>0$, for otherwise each summand in the left-hand side would be negative; therefore, the eigenvalues of $\lambda D$ are positive.
\end{proof}

\begin{example}
\label{ex:nm2}
Let $\mn_{m,2}$ denote the free nilpotent Lie algebra of step $2$ over $m$ generators. By Lemma~\ref{lemma:nikfree}, the Nikolayevsky derivation is
\[\frac{m}{2m-1}(\pi_1+2\pi_2).\]
Indeed, using~\eqref{eqn:dmsmalls}, for $s=2$,~\eqref{eqn:horrible} becomes
\[m(\lambda-1)+m(m-1)(2\lambda-1)=0,\]
i.e. $\lambda = \frac{m}{2m-1}$.
\end{example}

\begin{example}
Let $\mn_{m,3}$ denote the free nilpotent Lie algebra of step $3$ over $m$ generators. Using~\eqref{eqn:dmsmalls}, for $s=3$,~\eqref{eqn:horrible} becomes
\[m(\lambda-1)+m(m-1)(2\lambda-1)+m(m^2-1)(3\lambda-1)=0,\]
or equivalently $\lambda = \frac{m^2+m-1}{3m^2+2m-4}$.
Then the Nikolayevsky derivation is
\[ \frac{m^2+m-1}{3m^2+2m-4}(\pi_1+2\pi_2+3\pi_3).\]
\end{example}

\begin{lemma}
\label{lemma:n25notnice}
Let $\g$ be a nice nilpotent Lie algebra of step $s$. Let $f\colon \mn_{2,s}\to\g$ be the Lie algebra homomorphism mapping two generators of $\mn_{2,s}$ to two elements of a nice basis. Then, for the grading in~\eqref{eq:gradnds}, $\dim f(W_5)\leq 4$.

In particular, given a nice nilpotent Lie algebra of step $s\geq 5$, for any two  elements $X_1,X_2$ in a nice basis of $\mg$, the Lie subalgebra of $\mg$ spanned by $X_1$ and $X_2$ is not isomorphic to $\mn_{2,s}$.
\end{lemma}
\begin{proof}
Let $X_1,X_2\in\g$ be elements of a nice basis of $\mg$, and assume that they are the respective image under $f$ of generators $e_1,e_2$ of $W_1$. Notice that if $s\leq 4$, then $W_5=0$ and the result is obvious.

Let us assume that $s>4$. The image of $f$ is generated by the images of a Hall basis, as given in Example~\ref{example:hall}. In particular, $W_5$ is spanned by elements of length $5$ in the Hall basis, namely, $e_9, \ldots, e_{14}$ (see Example~\ref{example:hall}). We claim that $f$ has a kernel of dimension at least $2$ when restricted to $W_5$. This, together with the fact that $\dim W_5=6$, implies that $\dim f(W_5)\leq 4$.

Consider $X:=f([[e_2,e_1],e_1])=[[X_2,X_1],X_1]$ and  $Y:=f([[e_2,e_1],e_2])=[[X_2,X_1],X_2]$; these are elements in $f(W_3)$ and belong to the nice basis. Moreover, each of the following sets
\begin{align*}
&\left\{[[X,X_1],X_2]=f(e_{10}),\ [[X_1,X_2],X]=f(e_{13}),\ [[X,X_2],X_1]\right\},\\
&\left\{[[Y,X_1],X_2]=f(e_{11}),\ [[X_1,X_2],Y]=f(e_{14}),\ [[Y,X_2],X_1]\right\}
\end{align*}
spans a subspace of dimension $\leq 1$, due to the nice condition joint with the Jacobi identity.
Therefore, $\ker f|_{W_5}\geq 2$ as we wanted to prove.

For the second part, notice that the Lie subalgebra $\mh$ spanned by $X_1,X_2$ is also a nice Lie algebra. If $\mh$ is isomorphic to $\mn_{2,s}$, then there is an isomorphism $f\colon\mn_{2,s}\to \mh$ sending a generator set $e_1,e_2$ to $X_1,X_2$, respectively. Since the kernel is trivial, this contradicts the above for $s\geq 5$.
\end{proof}

In order to study the nice condition on free nilpotent Lie algebras and their cotangents, we need to introduce one more technical result.

\begin{lemma}
\label{lemma:estimate}
For any $s\geq 4$, $m\geq 2$:
\begin{equation}
\label{eq:estimate}
d_m(s)+2\sum_{k=1}^{[\frac{s+1}{2}]} kd_m(k)<m^s.
\end{equation}
\end{lemma}
\begin{proof}
By~\eqref{eqn:recurrenceformula}, the left hand side of~\eqref{eq:estimate} equals
\begin{equation}\label{eq:ddsp}
d_m(s)+2\sum_{k=1}^{[\frac{s+1}{2}]} kd_m(k)=\frac1{s}\left( m^s- \sum_{\begin{smallmatrix}k<s\\k|s\end{smallmatrix}} kd_m(k)\right)+2\sum_{k=1}^{[\frac{s+1}{2}]}kd_m(k).
\end{equation}

In addition, it is easy to see that for every $s>1$,  the following inclusion holds
\[
\{k\in\N\;:\;k<s,\ k|s\}\subseteq \left\{k\in\N\;:\;k\leq \left[\frac{s+1}{2}\right]\right\}.
\]
Indeed, if $1\leq k<s$, $k|s$, then $s=k\ell$ with $\ell\geq 2$ and thus $k\leq s/2$. Therefore, $k\leq [s/2]\leq [\frac{s+1}{2}]$.

Using this observation in~\eqref{eq:ddsp}, we get
\begin{equation}
d_m(s)+2\sum_{k=1}^{[\frac{s+1}{2}]} kd_m(k)
=\frac1{s}m^s+2\sum_{ \begin{smallmatrix}k=1\\k\mid s\end{smallmatrix}   }^{[\frac{s+1}{2}]} (k-\frac{k}{2s})d_m(k)+2\sum_{ \begin{smallmatrix}k=1\\k\nmid s\end{smallmatrix}   }^{[\frac{s+1}{2}]}kd_m(k).
\end{equation}
Therefore,~\eqref{eq:estimate} is equivalent to the following
\begin{equation}\label{eq:claim}
\sum_{\begin{smallmatrix}k=1\\
k|s\end{smallmatrix}}^{[\frac{s+1}{2}]}(k-\frac{k}{2s})d_m(k)+\sum_{\begin{smallmatrix}k=1\\
k \nmid s\end{smallmatrix}}^{[\frac{s+1}{2}]} kd_m(k)<\frac{(s-1)}{2s}m^s.
\end{equation}
We will show that~\eqref{eq:claim} holds for any $m\geq 2$ and $s\geq 4$. We clearly have
\[\sum_{\begin{smallmatrix}k=1\\
k|s\end{smallmatrix}}^{[\frac{s+1}{2}]}(k-\frac{k}{2s})d_m(k)+\sum_{\begin{smallmatrix}k=1\\
k \nmid s\end{smallmatrix}}^{[\frac{s+1}{2}]} kd_m(k)
\leq
\sum_{k=1}^{[\frac{s+1}{2}]}kd_m(k).\]
For $s=4$, \eqref{eq:claim} follows from
\[\sum_{k=1}^{[\frac{4+1}{2}]}kd_m(k)=m+m(m-1)\leq \frac{4-1}{8}m^4.\]
For $s=5$, it follows from
\[\sum_{k=1}^{[\frac{5+1}{2}]}kd_m(k)=m+m(m-1)+m(m^2-1)\leq \frac{5-1}{10}m^5.\]
For $s\geq 6$, we  have
\[[\frac{s+1}2]+1=s-2-\lceil\frac{s-7}2\rceil\leq s-2.\]
By~\eqref{eqn:recurrenceformula}, $kd_m(k)\leq m^k$; we obtain
\begin{equation}\label{eq:mult}
\sum_{k=1}^{[\frac{s+1}{2}]}kd_m(k)
\leq \sum_{k=1}^{[\frac{s+1}{2}]}m^k=\frac{m^{[\frac{s+1}{2}]+1}-1}{m-1}
\leq \frac{m^{s-2}-1}{m-1}< m^{s-2}\leq \frac14m^s\leq
\frac{s-1}{2s}m^s.\qedhere
\end{equation}
\end{proof}

\begin{theorem}
\label{thm:rude}
Let $\mn_{m,s}$ be the free $s$-step nilpotent Lie algebra on $m$ generators. The following are equivalent:
\begin{enumerate}
\item $\mn_{m,s}$ is nonnice;
\item $T^*\mn_{m,s}$ is nonnice;
\item $m=2$ and $s\geq 5$ or $m\geq 3$ and $s\geq 3$.
\end{enumerate}
\end{theorem}
\begin{proof}
It is easy to see that $\mn_{m,s}$ is nice for $s=1,2$. Similarly, for  $\mn_{2,3}$, $\mn_{2,4}$ a nice basis can be constructed by taking the first elements of the basis constructed in Example~\ref{example:hall}. The corresponding cotangents are also nice.

Now assume that either $m=2$ and $s\geq 5$ or $m\geq 3$ and $s\geq 3$. We need to prove that $\mn_{m,s}$ and its cotangent are nonnice.

Consider first the case of $\mn_{m,s}$. In view of Lemma~\ref{lemma:nikfree}, we know that $W_1$ in~\eqref{eq:gradnds} is an eigenspace for the Nikolayevski derivation. This implies that, if  $\mn_{2,s}$ is nice, then any nice basis contains a basis of $W_1$. Let $\{X_1,X_2\}$ denote these two elements, which are also a generator set of $\mn_{2,s}$. Hence, they span $\mn_{2,s}$ contradicting Lemma~\ref{lemma:n25notnice} as soon as $s\geq 5$.

For $m\geq 3$, we claim that
\begin{equation}
\label{eqn:W1W1W1}
 \dim [[W_1,W_1],W_1]-\frac16m(-4+3m+m^2)>0.
\end{equation}
Indeed, on the one hand,
\[\dim [[W_1,W_1],W_1]=\dim W_3=\frac13m(m^2-1),\]
and, on the other hand,
\[\frac13m(m^2-1)-\frac16m(-4+3m+m^2)=\frac16m(m-1)(m-2)>0.\]
This implies~\eqref{eqn:W1W1W1}, so Lemma~\ref{lemma:freenotnice} shows that $\mn_{m,s}$ is not nice.

For $T^*\mn_{m,s}$, if $W_1$ is an eigenspace of the Nikolayevski derivation, the same arguments used for $\mn_{m,s}$ (both $m=2$ and $m\geq 3$) apply to the cotangents. Suppose for a contradiction that $W_1$ is not an eigenspace. Then, reasoning as in the proof of Lemma~\ref{lemma:nonnice}, one can see that $W_1$ is contained in some eigenspace of the form $W_1+W_n^*$, with $1\leq n\leq s$, and $\lambda (1+n)=2$, where $\lambda$ is as in Lemma~\ref{lemma:nikfree}. In this notation, equation~\eqref{eqn:horrible} is equivalent to
\begin{equation}
 \label{eqn:thehorror...thehorror...}
\sum_{k=1}^s kd_m(k)(2k-n-1)=0
\end{equation}
For $s=3$ this equation is
\[0=m(2-n-1)+m(m-1)(4-n-1)+m(m^2-1)(6-n-1).\]
It is straightforward to see that for any $n=1,2,3$ and $m\geq 3$ the equation above cannot hold.

For $s\geq 4$,  we can write
\begin{equation*}
\sum_{k=1}^{[\frac{s+1}{2}]} kd_m(k)(n+1-2k)=\sum_{k=[\frac{s+3}{2}]}^{s} kd_m(k)(2k-n-1).
\end{equation*}
Adding $\sum_{k<s,\ k|s} kd_m(k)(2s-n-1)$ to both sides,  and using on the right hand side that the term added equals $(m^s-sd_m(s))(2s-n-1)$ by~\eqref{eqn:recurrenceformula}, we get
\begin{equation}\label{eq:sumfrac}
\sum_{k=1}^{[\frac{s+1}{2}]} kd_m(k)(n+1-2k)+\sum_{k<s,\ k|s} kd_m(k)(2s-n-1)= \sum_{k=[\frac{s+3}{2}]}^{s-1} kd_m(k)(2k-n-1)+m^s(2s-n-1).
\end{equation}
Since every $k<s$ such that $k|s$ must be less than or equal to $[\frac{s+1}{2}]$, the left hand side in~\eqref{eq:sumfrac} satisfies
\begin{equation}\label{eq:lhs}
 \sum_{k=1}^{[\frac{s+1}{2}]} kd_m(k)(n+1-2k)+\sum_{k<s,\ k|s} kd_m(k)(2s-n-1)\leq \sum_{k=1}^{[\frac{s+1}{2}]} kd_m(k)(2s-2k)\leq 2(s-1)\sum_{k=1}^{[\frac{s+1}{2}]} kd_m(k).
\end{equation}
Moreover, since $1\leq n\leq s$, the right hand side in~\eqref{eq:sumfrac} verifies
\begin{equation}\label{eq:rhs}
m^s(s-1)\leq m^s(2s-n-1)\leq \sum_{k=[\frac{s+3}{2}]}^{s-1} kd_m(k)(2k-n-1)+m^s(2s-n-1).
\end{equation}
Thus, \eqref{eq:sumfrac}-\eqref{eq:rhs} imply
\[
m^s\leq 2\sum_{k=1}^{[\frac{s+1}{2}]} kd_m(k),
\]
contradicting Lemma~\ref{lemma:estimate}. Therefore, $W_1$ is always an eigenspace of $\tilde{N}$ and the result follows.
\end{proof}
\begin{remark}
Given a nice Lie algebra $\g$, the cotangent $T^*\g$ is nice. We do not know whether the converse holds in general, though Theorem~\ref{thm:rude} shows this to be true for the special case of free nilpotent Lie algebras.
\end{remark}

\begin{theorem}
For any $s\geq 3$, there exists an $s$-step nilpotent Lie algebra which is irreducible, not nice and admitting ad-invariant metrics.
\end{theorem}
\begin{proof}
For any $m\geq 2$, $s\geq 3$, $T^*\mn_{m,s}$ admits an ad-invariant metric; it is nonnice because of Theorem~\ref{thm:rude} and irreducible by Lemma \ref{lemma:cotangentirreducible}.
Moreover, one can easily check that $T^*\mn_{m,s}$ is $s$-step nilpotent.
\end{proof}

We point out that there exist nonnice Lie algebras of step two, but they are not known explicitly (see \cite[Example~3]{Nikolayevsky}), and the methods of this paper do not seem sufficient to determine whether they give rise to nonnice, irreducible cotangents.

\bigskip

\noindent \textbf{Funding details:}
D. Conti acknowledges the MIUR Excellence Department Project awarded to the Department of Mathematics, University of Pisa, CUP I57G22000700001. D.~Conti and F.A.~Rossi acknowledge GNSAGA of INdAM and the PRIN project n. 2022MWPMAB ``Interactions between Geometric Structures and Function Theories''. F.A. Rossi acknowledges the INdAM-GNSAGA project CUPE55F22000270001 ``Curve algebriche e loro applicazioni'' and the Young Talents Award of Universit\`{a} degli Studi di Milano-Bicocca joint with Accademia Nazionale dei Lincei. V.~del Barco supported by FAEPEX-UNICAMP grant 2566/21; FAPESP grants 2021/09197-8 and 2023/15089-9.


\bibliographystyle{plain}

\bibliography{nice.bib}

\small\noindent  D.~Conti: Dipartimento di Matematica, Università di Pisa, largo Bruno Pontecorvo 6, 56127 Pisa, Italy.\\
\texttt{diego.conti@unipi.it}\\\medskip

\small\noindent V.~del Barco: Instituto de Matemática, Estatística e Computação Científica, Universidade Estadual de Campinas, Rua Sergio Buarque de Holanda, 651, Cidade Universitaria Zeferino Vaz, 13083-859, Campinas, São Paulo, Brazil.\\
\texttt{delbarc@unicamp.br}\medskip

\small\noindent F.A.~Rossi: Dipartimento di Matematica e Informatica, Universit\`a degli studi di Perugia, via Vanvitelli 1, 06123 Perugia, Italy.\\
\texttt{federicoalberto.rossi@unipg.it}

\end{document}